\documentclass[12pt,twoside, english]{amsart}

\usepackage{hyperref}
\usepackage{amsthm,thmtools,xcolor}

\usepackage{verbatim}
\usepackage{amsmath}
\usepackage{bm}
\usepackage{a4wide}

\usepackage{babel,latexsym,amsmath,amsthm,eucal}
\usepackage[utf8]{inputenc}

\usepackage[T1]{fontenc}
\usepackage{times}
\usepackage{amssymb,latexsym}

\usepackage{amssymb,latexsym}
\usepackage{enumerate}

\newcommand{\dbar}{\ensuremath{\overline\partial}}

\makeatletter
\newcommand{\sumprime}{\if@display\sideset{}{'}\sum%
            \else\sum'\fi}
\makeatother

\newcommand{\ip}[2]{\ensuremath{\left\langle #1, #2\right\rangle}}

\newcommand{\rbrac}[1]{\ensuremath{  \left( #1 \right)  }}

\newcommand{\fr}[2]{\ensuremath{  \frac{#1}{#2}  }}

\newcommand{\bt}[1]{\textbf{#1}}

\newcommand{\id}[1]{\ensuremath{\bt{1}_{#1}}}

\newcommand{\ita}[1]{\emph{#1}}

\newcommand{\set}[1]{\ensuremath{\mathbb{#1}}}

\newcommand{\sub}{\ensuremath{\subseteq}}

\newcommand{\norm}[1]{\ensuremath{\left| \left |#1 \right|\right|}}

\begin{document}

\numberwithin{equation}{section}

\newtheorem{theorem}{Theorem}[section]
\newtheorem{proposition}[theorem]{Proposition}
\newtheorem{conjecture}[theorem]{Conjecture}
\def\theconjecture{\unskip}
\newtheorem{corollary}[theorem]{Corollary}
\newtheorem{lemma}[theorem]{Lemma}
\newtheorem{observation}[theorem]{Observation}
\newtheorem{definition}{Definition}
\numberwithin{definition}{section} 
\newtheorem{remark}{Remark}
\def\theremark{\unskip}
\newtheorem{kl}{Key Lemma}
\def\thekl{\unskip}
\newtheorem{question}{Question}
\def\thequestion{\unskip}
\newtheorem{example}{Example}
\def\theexample{\unskip}
\newtheorem{problem}{Problem}

%

\bigskip

\address{DEPARTMENT OF MATHEMATICAL SCIENCES, NORWEGIAN UNIVERSITY OF SCIENCE AND TECHNOLOGY, NO-7491 TRONDHEIM, NORWAY}
\email{tai.t.h.nguyen@ntnu.no}
\email{xu.wang@ntnu.no}

\title[On a remark of Ohsawa]{On a remark by Ohsawa related to the Berndtsson-Lempert method for $L^2$-holomorphic extension}

 \author{Tai Terje Huu Nguyen and Xu Wang}
\date{\today}

\begin{abstract} 
We utilize the Legendre-Fenchel transform and weak geodesics for plurisubharmonic functions to construct a weight function that can be used in the Berndtsson-Lempert method, to give an Ohsawa-Takegoshi extension type of result. Theorem 4.1 and Theorem 0.1 in \cite{OT2017} (Theorem \ref{Theorem A} and \ref{Theorem B} below) follow as two special cases of this result, thus answering affirmatively a question posed by Ohsawa in remark 4.1 in \cite{OT2017}, on the Berndtsson-Lempert method.
\end{abstract}

\maketitle

\tableofcontents

\section{Introduction}  

In \cite{OT2017} Ohsawa gave a new proof of two theorems of Guan-Zhou (see \cite{GZ0, YLZ, G} for further details and related results), Theorem  \ref{Theorem A} and \ref{Theorem B} below, and asked in a remark (remark 4.1 in \cite{OT2017}) whether a proof using the Berndtsson-Lempert method in \cite{BoLempOT} can be had.

\medskip

\begin{theorem}[Theorem 4.1 in \cite{OT2017}]\label{Theorem A}
{\it Let $\Omega$ be a pseudoconvex domain in $\set{C}^n$, $\Omega':=\{(z',z_n)\in \Omega:z_n=0\}$, $\phi$ a plurisubharmonic function on $\Omega$, $f\in \mathcal{O}(\Omega')$ a holomorphic function on $\Omega'$, and let $\alpha>0$. Then there exists $\tilde{f}\in \mathcal{O}(\Omega)$ a holomorphic function on $\Omega$ such that $\tilde{f}|_{\Omega'}=f$ and
\begin{align*}
\int_{\Omega}|\tilde{f}|^2e^{-\phi-\alpha|z_n|^2}\leq \fr{\pi}{\alpha}\int_{\Omega'}|f|^2e^{-\phi},
\end{align*}}where integration is with respect to the Lebesgue measure.
\end{theorem}

\begin{theorem} [Theorem 0.1 in \cite{OT2017}]\label{Theorem B} {\it With the same assumptions and notation as in Theorem \ref{Theorem A}, there exists $\tilde{f}\in \mathcal{O}(\Omega)$ such that $\tilde{f}|_{\Omega'}=f$ and satisfying the estimate
\begin{align*}
\int_{\Omega}\fr{|\tilde{f}|^2e^{-\phi}}{(1+|z_n|^2)^{1+\alpha}}\leq \fr{\pi}{\alpha}\int_{\Omega'}|f|^2e^{-\phi}.
\end{align*}}
\end{theorem}
In this paper we answer Ohsawa's question affirmatively by proving an extension theorem with estimate using the Berndtsson-Lempert method. This is Theorem \ref{main main theorem} below and the main theorem of this paper. A key ingredient in our proof is the construction of a weight function $\theta$ using the Legendre-Fenchel transform and weak geodesics for plurisubharmonic functions (see the appendix in Section \ref{Appendix}) to which we can apply the Berndtsson-Lempert technique.

\begin{theorem}[Main Theorem]\label{main main theorem}
{\it Let $\Omega$ be a pseudoconvex domain in $\set{C}^n$,  $1\leq k \leq n$ an integer,
$$
\Omega':=\{(z',z'')\in \Omega: z''=0\},  \ \ \ z'':=(z_{n-k+1}, \cdots, z_n),
$$
$\phi$ a plurisubharmonic function on $\Omega$, $f\in \mathcal{O}(\Omega')$ a holomorphic function on $\Omega'$, and $\sigma$ a convex increasing function on $\mathbb R^n$. Then there exists $\tilde{f}\in \mathcal{O}(\Omega)$ a holomorphic function on $\Omega$ such that $\tilde{f}|_{\Omega'}=f$ and 
\begin{align*}
\int_{\Omega}|\tilde{f}|^2e^{-\phi-\sigma(\ln|z''|)}\leq \int_{\Omega''} e^{-\sigma(\ln|z''|)} \int_{\Omega'}|f|^2e^{-\phi}, 
\end{align*}
where we write $\ln |z''|:=(\ln|z_{n-k+1}|, \cdots, \ln|z_n|)$, $\Omega'':=\{z''\in \mathbb C^k: \ln |z''| \in 
\Omega''_{\mathbb R}\}$, $ \Omega''_{\mathbb R}$ denotes the convex hull of $\{\ln |z''| : z\in \Omega\}$}, and integration is with respect to the Lebesgue measure.
\end{theorem}

In the case that $\sigma$ depends only on $|z''|^2$, Theorem \ref{main main theorem} reduces to \cite[Theorem 1.7]{YLZ}. It is also very likely that the main theorem in \cite{GZ0} implies Theorem \ref{main main theorem}. So our contribution is not the theorem itself, but rather a method of proof using the Berndtsson-Lempert method.

\section{The Berndtsson-Lempert method in a simple setting}\label{BL method}
We start by recalling quickly the Berndtsson-Lempert method for $L^2$-holomorphic extension (\cite{BoLempOT}) in the simple setting of domains in $\set{C}^n$ as it will apply to us. Specifically, we consider the following set-up. Let $\Omega\sub \set{C}^n$ be a pseudoconvex domain and let $\phi$ be a plurisubharmonic function on $\Omega$. For $z\in \Omega$, we write $z=(z',z_n)$. Let $\Omega':=\{z\in \Omega:z_n=0\}$, suppose that $f\in \mathcal{O}(\Omega')$ is a given holomorphic function on $\Omega'$, and suppose that $|z_n|<1$ on $\Omega$. We are interested in finding a holomorphic extension $\tilde{f}\in \mathcal{O}(\Omega)$ of $f$ defined on all of $\Omega$ for which we have a good weighted $L^2$ estimate of the form (\cite{OT})
\begin{align*}
\int_{\Omega}|\tilde{f}|^2e^{-\phi}\leq C\int_{\Omega'}|f|^2e^{-\phi},
\end{align*}
where integration is with respect to the Lebesgue measure, and $C$ is some universal constant. The Berndtsson-Lempert method demonstrated in \cite{BoLempOT} is the following. Fix a positive constant $j>1$ and consider the following plurisubharmonic function 
\begin{align}\label{eq:bl0}
\theta: \mathbb C \times \Omega \to \mathbb R , \quad (\tau, z)\mapsto \theta(\tau,z):=\theta^{{\rm Re}\,\tau}(z):=\phi+ j \max\{\ln|z_n|^2-{\rm Re}\, \tau, 0\}.
\end{align}
For each $t:={\rm Re}\,\tau \leq 0$, $\theta^t$ provides a weight function for a weighted $L^2$ inner product $\ip{\cdot}{\cdot}_t$ and the corresponding induced weighted $L^2$ norm $\norm{\cdot}_t$, given by
\begin{align*}
\ip{u}{v}_{t}&:=\int_{\Omega}u\bar{v}e^{-\theta^{t}},\quad ||u||^2_t:= \ip{u}{u}_{t}.
\end{align*}
Since $|z_n|<1$ on $\Omega$, we have $\theta^{0}=\phi$. Let $f_{t}$ denote the holomorphic extension of $f$ with minimal $\norm{\cdot}_t$-norm. Then we can write  \cite[section 3]{BoLempOT}
\begin{align*}
\norm{f_t}^{2}_{t}&=\sup_{g\in C^{\infty}_{0}(\Omega')}\fr{|\xi_{g}(F)|^2}{\norm{\xi_g}^2_{t}},\quad \xi_g(F):=\int_{\Omega'} F\bar{g}e^{-\phi},
\end{align*}
where $F$ denotes \ita{any} holomorphic extension of $f$, and where $\norm{\xi_{g}}_{t}^2$ denotes the (squared) dual norm of the continuous linear functional $\xi_{g}$, given by 
\begin{align*}
\norm{\xi_{g}}_{t}^2&=\sup_{\norm{F}_{t}=1}|\xi_{g}(F)|^2.
\end{align*}
The crux of the argument is that by complex Brunn-Minkowski theory (see e.g. \cite{BoCBMT, BoCBMVectBundles, Bern06, Bern09}), under the assumption that $\theta$ is plurisubharmonic, we have that $t\mapsto \norm{\xi_g}_{t}^2$ is log-convex. Hence $t\mapsto e^{t}\norm{\xi_g}_{t}^2$ is log-convex. Suppose that this is also bounded as $t\to-\infty$. Then $t\mapsto t+\ln(\norm{\xi_g}_t^2)$ is convex and bounded as $t\to-\infty$, and hence increasing. Taking the exponential, we infer that $t\mapsto e^{t}\norm{\xi_g}_{t}^2$ is increasing, which implies that $t\mapsto e^{-t}\norm{f_t}_t^2$ is decreasing. Thus, we get the inequality
\begin{align*}
\norm{f_0}^2_{0}&\leq \lim_{t\to-\infty}e^{-t}\norm{f_t}^2_t.
\end{align*}Since $f_t$ is the extension of $f$ with \ita{minimal} $\norm{\cdot}_t$-norm, the right-hand side limit is bounded by the limit of $e^{-t}\norm{F}_{t}^2$ as $t\to-\infty$ with $F$ \ita{any} extension as above. Writing $f_0:=\tilde{f}$ and recalling that we are assuming $\theta^{0}=\phi$, we therefore find
\begin{align*}
\int_{\Omega}|\tilde{f}|^2e^{-\phi}&\leq \liminf_{t\to-\infty}e^{-t}\norm{F}_t^2=\liminf_{t\to-\infty}e^{-t}\int_{\Omega}|F|^2e^{-\theta^{t}}.
\end{align*} 
By \cite[Lemma 3.2, Lemma 3.3]{BoLempOT}, as $j\to \infty$, the right-hand side limit converges to $\pi\int_{\Omega'}|f|^2e^{-\phi}$. Thus, we get an extension $\tilde{f}$ with the following estimate:
\begin{align*}
\int_{\Omega}|\tilde{f}|^2e^{-\phi}&\leq \pi\int_{\Omega'}|f|^2e^{-\phi}.
\end{align*}
This estimate is known to be optimal (we have equality in the case $\Omega=\Omega' \times \{|z_n|<1\}$ and $\phi$ does not depend on $z_n$) and was first proved in \cite{Bl0, GZ0}.

\section{A solution to Ohsawa's question}

In this section we apply the Berndtsson-Lempert method in Section \ref{BL method} to a weight function $\theta$ that is different from the one in \eqref{eq:bl0} to prove Theorem \ref{main theorem} below.
Notice that taking
$\sigma=\alpha e^{2x}$ or $\sigma=(1+\alpha)\ln(1+e^{2x})$ we get $\mathrm{L}=\pi/\alpha$ in Theorem \ref{main theorem}. Hence Theorem \ref{Theorem A} and \ref{Theorem B} follow.  Our construction of $\theta$ is based on a Legendre-Fenchel transform approach to weak geodesics for plurisubharmonic functions (see the appendix in Section \ref{Appendix}).

\begin{theorem}\label{main theorem}
Let $\Omega\sub \set{C}^n$ be pseudoconvex and write for $z\in \Omega$, $z=(z',z_n)$. Let $\phi$ be a plurisubharmonic function in $\Omega$, $\Omega':=\{z\in \Omega:z_n=0\}$, and  $f\in \mathcal{O}(\Omega')$ a holomorphic function on $\Omega'$. Put $x:=\ln|z_n|$ and let $\sigma=\sigma(x)$ be a 
convex and increasing function in $x$ with the property that $\mathrm{L}:=\displaystyle \int_{\mathbb C} e^{-\sigma(\ln |w|)} <\infty$. Then there exists $\tilde{f}\in \mathcal{O}(\Omega)$ a holomorphic function on $\Omega$ such that $\tilde{f}|_{\Omega'}=f$ and
\begin{align}
\int_{\Omega}|\tilde{f}|^2e^{-\phi-\sigma(\ln|z_n|)}\leq \mathrm{L}\int_{\Omega'}|f|^2e^{-\phi}.\label{final estimate}
\end{align}
\end{theorem}

\begin{proof}  Fix $c\in \mathbb R$ and consider the weight function $\theta$ defined by
$$
\theta(\tau, z):= \theta^{{\rm Re}\,\tau}(z) := \phi(z)+ \psi^{{\rm Re}\,\tau}(\ln|z_n|), \ \ \  t={\rm Re}\,\tau >0,
$$
where
\begin{align}
\psi^t(x):=\psi(t,x)&:=t\sigma\rbrac{\fr{x-c}{t}+c}.\label{4}
\end{align}
One may directly verify that (see also the appendix in Section \ref{Appendix} for another approach) $\psi$ is convex in $(t,x)$ and increasing with respect to $x$. Since $\phi$ is assumed to be plurisubharmnic, we know that $\theta$ defined above is plurisubharmonic in $(\tau, z)$. By the Berndtsson-Lempert method, using the same notation as in Section \ref{BL method}, it follows that
$$
\rho_c(t):=\ln ||\xi_g||^2_t-2c(t-1)
$$ 
is convex as a function of $t$. Hence for all $t\in [1/2, 1)$, 
$$
\frac{\rho_c(1)-\rho_c(t)}{1-t} \geq  \frac{\rho_c(1)-\rho_c(1/2)}{1-1/2}.
$$
Observe that $\rho_c(1)= \ln ||\xi_g||^2_1$ does not depend on $c$. Thus,
$$
\ln ||\xi_g||^2_1 \geq \rho_c(t) + 2(1-t) \left(\ln ||\xi_g||^2_1-\rho_c(1/2) \right), 
$$ for all $t\in [1/2,1)$ \ita{and} for all $c\in \set{R}$.
Now let
$$
c:= \frac{-1}{(1-t)^2},
$$
and let $t \uparrow 1$. Lemma \ref{le:1} below implies that
\begin{align}
\ln ||\xi_g||^2_1 \geq  \limsup_{t\uparrow 1}  \rho_c(t).\label{convex ineq}
\end{align}
Hence, by \eqref{convex ineq}, we have a holomorphic extension $\tilde f$ with the following estimate
\begin{align*}
\int_{\Omega}|\tilde{f}|^2e^{-\phi-\sigma(\ln|z_n|)}&\leq  \liminf_{t \uparrow 1} e^{2c(t-1)}  \int_{\Omega}  |F|^2 e^{-\phi-\psi^t(\ln|z_n|)}  \\
&= \liminf_{t \uparrow 1}  e^{2/(1-t)}  \int_{\Omega}|F|^2  e^{-\phi-t \sigma\left( \frac{\ln|z_n|+\frac1{1-t}}{t}\right)} ,
\end{align*}
where $F$ is any arbitrary holomorphic extension. From the change of variables
$$
z_n= e^{\frac{1}{t-1}} w, \ \ \Omega_t:= \{(z', w) \in \mathbb C^n: (z', e^{\frac{1}{t-1}} w) \in \Omega\}, 
$$
and the definition of $\mathrm{L}$, we find that letting $t\uparrow 1$, 
$$
e^{2/(1-t)}  \int_\Omega |F|^2  e^{-\phi-t \sigma\left( \frac{\ln|z_n|+\frac1{1-t}}{t}\right)} = \int_{\Omega_t} |F(z', e^{\frac{1}{t-1}} w)|^2 e^{-\phi\left(z', e^{\frac{1}{t-1}} w\right) - t \sigma \left(\frac{\ln |w|}{ t}\right)}
$$
converges to the right hand side of \eqref{final estimate}. Hence the theorem follows.
\end{proof}

\begin{lemma}\label{le:1} Let $c:= \frac{-1}{(1-t)^2}$. Then $
\lim_{t \uparrow 1} (1-t) \rho_c(1/2) =0$.
\end{lemma}

\begin{proof} It suffices to show that $\limsup_{c\to -\infty} \rho_c(1/2) <\infty$. That is, 
$$
 \limsup_{c\to 
 \infty}  \sup_{F \in \mathcal O(\Omega)}   \frac{\left|\displaystyle\int_{\Omega'} F\bar{g}e^{-\phi}\right|^2}{e^{-c}\displaystyle\int_{\Omega} |F|^2  e^{-\phi-\frac12 \sigma\left( 2\ln|z_n|-c\right)} } <\infty.
$$
This now follows from the change of variables $z_n= e^{c/2} w$ and the definition of $\mathrm{L}$.
\end{proof}

\textbf{Remark}: The change of variables argument in the proof of Theorem \ref{main theorem} also suggests to use another weight function, with
\begin{equation}\label{eq:new1}
\psi^t(x)=\sigma(x-t), \ \  t \in \mathbb R.
\end{equation}
In fact, as we shall see next, the proof of  Theorem \ref{main theorem} is simpler if we use this new weight function. Nevertheless, we still wish to include the proof given above since it is our first example of the use of weak geodesics in the Berndtsson-Lempert method. 

\subsection[Proof of the main theorem]{Proof of the main theorem (Theorem \ref{main main theorem})}
\begin{proof}[\ita{Since $\Omega\subset \mathbb C^{n-k} \times \Omega''$, one may assume that $\sigma=\infty$ outside $\Omega''_{\mathbb R}$}] Also we may assume that the right-hand side in the estimate in the theorem is finite, else there is nothing to prove.
Replace $\psi^t$ in \eqref{4} by
\begin{equation}\label{eq:new2}
\psi^t(x)=\sigma(x-t), \ \  t \in \mathbb R, \ \  x-t:=(x_1-t, \cdots, x_k-t).
\end{equation}
It is clear that $\psi^t(x)$ in convex in $(t,x)$ and increasing in $x$. Hence, the corresponding weight function 
$$
\theta(\tau, z)= \phi(z)+ \psi^{{\rm Re}\,\tau}(\ln|z_n|)
$$
is plurisubharmonic in $(\tau, z)$, and the Berndtsson-Lempert method applies. Thus, we know that $t\mapsto \ln \norm{\xi_g}_t^2-2kt$ is convex. Moreover, the change of variables $z''= e^t w$ gives, 
\begin{align}\label{eq:rotation}
e^{2kt} \int_{\Omega}  |F|^2 e^{-\phi-\sigma(\ln|z''|-t)} & =   \int_{\{(z', e^t w) \in \Omega\}} |F(z', e^t w)|^2 e^{-\phi(z', e^t w)-\sigma(\ln|w|)}, 
\end{align} 
where $F$, as before, is any arbitrary fixed extension of $f$. From \eqref{eq:rotation} we infer that $\ln \norm{\xi_g}_t^2-2kt$ is also bounded near $t=-\infty$. Hence, $t\mapsto \ln \norm{\xi_g}_t^2-2kt$ is increasing and there exists a holomorphic extension $\tilde f$ with 
\begin{align*}
\int_{\Omega}|\tilde{f}|^2e^{-\phi-\sigma(\ln|z''|)} & \leq 
 \liminf_{t\to -\infty}   \int_{\{(z', e^t w) \in \Omega\}} |F(z', e^t w)|^2 e^{-\phi(z', e^t w)-\sigma(\ln|w|)}   \\
  &  \leq \int_{\mathbb C^k} e^{-\sigma(\ln|w|)} \int_{\Omega'}|f|^2e^{-\phi}  \\
   \ & = \int_{\Omega''} e^{-\sigma(\ln|w|)} \int_{\Omega'}|f|^2e^{-\phi},
\end{align*} 
where in the last equality we have used that \text{$\sigma=\infty$ outside $\Omega''_{\mathbb R}$}.
This completes the proof. \end{proof}

\medskip

\textbf{Remark}: From \eqref{eq:rotation}, in the case that $\Omega$ is horizontally balanced (that is, $(z', z'')\in \Omega \Rightarrow (z', \tau z'') \in \Omega$ for all $\tau \in \mathbb C$ with $|\tau|<1$), the method above is equivalent to the one associated to the weight function
$$
\theta(\tau, z):=\phi(z', \tau z'')+\sigma(\ln |z''|).
$$
This construction has already been used in \cite[Section 2]{Bern20}. The advantage of \eqref{eq:new2} is that it also applies to general $\Omega$. 

\section{Appendix: weak geodesics in the space of toric plurisubharmonic functions}\label{Appendix}

By a toric plurisubharmonic function we mean a function of the following form
$$
\bm \psi(z):=\psi( \log|z_1|, 
\cdots, \log|z_n|)
$$
on $\mathbb C^n$, where $\psi$ is a convex increasing function on $\mathbb R^n$.  

\begin{definition} We call a family, say $\{\bm \psi^t\}_{0<t<1}$, of toric plurisubharmonic functions a weak geodesic \footnote{For usual weak psh geodesics on compact K\"ahler manfiolds, see \cite[Section 2]{BB}. The weak geodesic is also called generalized geodesic in \cite[Section 2.2]{Bern15}, and maximal psh segment or psh geodesic segment in \cite[page 5]{R}. See \cite{M} for the background.} if 
$$
\bm \Psi : (\tau ,z) \mapsto \bm\psi^{{\rm Re}\,\tau}(z), \ \ \tau\in \mathbb C_{(0,1)}:=\{\tau\in\mathbb C: 0<{\rm Re}\,\tau<1\} 
$$
is plurisubharmonic on $\mathbb C_{(0,1)} \times \mathbb C^n$ and $(i\partial\dbar \bm \Psi)^{n+1}=0$ on $\mathbb C_{(0,1)} \times (\mathbb C\setminus\{0\})^n$.
\end{definition}

\textbf{Remark}: Noticing that $\bm\Psi$ is locally bounded on $\mathbb C_{(0,1)} \times (\mathbb C\setminus\{0\})^n$, we know that $(i\partial\dbar \bm \Psi)^{n+1}$ is well defined. Moreover, by a local change of variables $z_j=e^{w_j}$, we know that $\{\bm \psi^t\}_{0<t<1}$ is a weak geodesic if and only if
$$
\psi(t,x):=\psi^t(x)
$$
is convex on $(0,1) \times \mathbb R^n$ and
$$
MA(\psi):=\det (D^2_{(t,x)} \psi) =0 
$$
on  $(0,1) \times \mathbb R^n$, where $D^2_{(t,x)} \psi$ denotes the real Hessian matrix of $\psi$ with respect to $(t,x)$. Hence, by Corollary 2.5 in \cite{BoConvex}, $\{\bm \psi^t\}_{0<t<1}$ is a weak geodesic if and only if the Legendre-Fenchel transform of $\psi(t,x)$ with respect to $x$ for fixed $t$, say 
$$
(\psi^{t})^{*}(\xi):= \sup_{x\in \mathbb R^n} x\cdot \xi -\psi^t(x),
$$
is an affine function in $t$. This suggest the following definition.

\begin{definition} We call a family, say $\{\bm \psi^t\}_{0 < t < 1}$, of toric plurisubharmonic functions, a weak geodesic segment if 
$$
(\psi^t)^*= t (\sigma^1)^*+(1-t) (\sigma^0)^* 
$$
with $\sigma^1, \sigma^0$ convex increasing functions on $\mathbb R^n$. 
\end{definition}

Now let us consider the case $n=1$. Let  $\sigma^1:=\sigma$ in Theorem \ref{main theorem}. For $\sigma^0$ we take
\begin{align*}
\sigma^0:=\id{(-\infty, c]}(x)=\left\{\begin{array}{ccc}
0,&x \leq c\\
\infty,&x > c.
\end{array}\right.
\end{align*}
Then we have
$$
 (\sigma^0)^*(\xi)=\sup_{x\in \mathbb R} x \xi -\sigma^0(x) = \sup_{x \leq c}  x\xi =c \xi+ \sup_{x\leq 0} x\xi= \begin{cases}
 c\xi & \xi \geq 0 \\
\infty & \xi <0.
 \end{cases}
$$
Since $\sigma^1 =\sigma$ is also increasing, we know that $(\sigma^1)^*(\xi)=\sigma^*(\xi) =\infty$ when $\xi <0$. Hence
$\psi^t$ can be written as
$$
\psi^t(x) =\sup_{\xi \in \mathbb R} x\xi-t\sigma^*(\xi)-(1-t) c\xi = t  \left(\sup_{\xi \in \mathbb R} \frac{x-(1-t)c}{t}\,\xi-\sigma^*(\xi) \right),
$$
and $\sigma^{**} =\sigma$ gives $\psi(t,x)=t\sigma\rbrac{\fr{x-c}{t}+c}$. This
is precisely \eqref{4}. Note that \eqref{eq:new1} also gives a weak geodesic since
$$
\sup_{x\in \mathbb R}  x\xi -\sigma(x-t) = t \xi +\sigma^*(\xi)
$$
is affine in $t$.


\begin{thebibliography}{99}

\bibitem{BB} R. Berman and B. Berndtsson, {\it Convexity of the K-energy on the space of K\"ahler metrics and uniqueness
of extremal metrics}, J. Am. Math. Soc. {\bf 30} (2017), 1165--1196.

\bibitem{Bern06} B. Berndtsson, {\it Subharmonicity properties of the Bergman kernel and some other functions associated to pseudoconvex domains}, Ann. Inst. Fourier (Grenoble), {\bf 56} (2006), 1633--1662.

\bibitem{Bern09} B. Berndtsson, {\it Curvature of vector bundles associated to holomorphic fibrations}, Ann. Math. {\bf 169} (2009), 531--560.

\bibitem{BoConvex} B. Berndtsson, \ita{Convexity on the space of K\"ahler metrics}, Ann. Fac. Sci. Toulouse {\bf 22} (2013), 713--746.

\bibitem{Bern15} B. Berndtsson, {\it A Brunn-Minkowski type inequality for Fano manifolds and some uniqueness theorems
in K\"ahler geometry}, Invent. Math. {\bf 200} (2015), 149--200.

\bibitem{BoCBMT} B. Berndtsson, \ita{Real and complex Brunn-Minkowski Theory}, AMS Contemorary Mathematics Volume 681, 2017.

\bibitem{BoCBMVectBundles} B. Berndtsson, \ita{Complex Brunn-Minkowski theory and positivity of vector bundles}, arXiv:1807.05844v1.  

\bibitem{Bern20} B. Berndtsson, {\it Bergman kernels for Paley-Wiener spaces and Nazarov's proof of the Bourgain-Milman theorem}, 	arXiv:2008.00837.

\bibitem{BoLempOT} B. Berndtsson and L. Lempert, \ita{A Proof of the Ohsawa-Takegoshi theorem with sharp estimates}, J. Math.
Soc. Japan, {\bf 68} (2016), 1461--1472.

\bibitem{Bl0}  Z. Blocki,  {\it Suita conjecture and the Ohsawa-Takegoshi extension theorem}, Invent. Math., 193
(2013), 149--158.

\bibitem{G} Q. Guan, {\it A remark on the extension of $L^2$ holomorphic functions}, arXiv:1809.05876.

\bibitem{GZ0} Q. Guan and X. Zhou, {\it A solution of an $L^2$ extension problem with an optimal estimate and
applications}, Ann. of Math. (2), {\bf 181} (2015), 1139--1208.

\bibitem{M} T. Mabuchi, {\it Some symplectic geometry on compact K\"ahler manifolds}, Osaka J. Math. {\bf 24} (1987), 227--252.

\bibitem{OT2017} T. Ohsawa, \ita{On the extension of $L^2$ holomorphic functions VIII - a remark on a theorem of Guan and Zhou}, Internat. J. Math. {\bf 28} (2017), no. 9, 1740005, 12 pp.

\bibitem{OT} T. Ohsawa and K. Takegoshi, {\it On the extension of $L^2$-holomorphic functions}, Math. Z. {\bf 195} (1987), 197--204.

\bibitem{R} R. R\'emi, {\it Plurisubharmonic geodesics in spaces of
non-Archimedean metrics of finite energy}, arXiv:2012.07972.

\bibitem{YLZ} S. Yao, Z. Li and X. Zhou,  {\it On the optimal extension theorem and a question of Ohsawa}, Nagoya Mathematical Journal, 1-12, doi:10.1017/nmj.2020.34.





\end{thebibliography}
\end{document}